\documentclass{article}
\usepackage{amsfonts}
\usepackage{amssymb}
\usepackage{amsmath}
\usepackage{amsthm}
\usepackage[dvips]{graphics}

\newtheorem{theorem}{Theorem}
\newtheorem{proposition}[theorem]{Proposition}
\newtheorem*{prop}{Proposition}
\newtheorem{lemma}[theorem]{Lemma}

\begin{document}
   \title{Disjointness of representations arising in harmonic analysis
    on the infinite-dimensional unitary group}
 \date{11 July}
 \author{Vadim Gorin\thanks{e-mail: vadicgor@rol.ru} \\
       Moscow State University.}
\maketitle

\begin{abstract}
 We prove pairwise disjointness of representations $T_{z,w}$ of the infinite-dimensional unitary group.
 These representations provide a natural generalization of the regular representation for the case of "big" group $U(\infty)$.
 They were introduced and studied by G.Olshanski and A.Borodin.
 Disjointness of the representations can be reduced to disjointness of certain probability measures on the space of paths in
 the Gelfand-Tsetlin graph. We prove the latter disjointness using probabilistic and combinatorial methods.
\end{abstract}

\section*{Introduction}

The aim of harmonic analysis on a group is to study regular and
another natural representations connected with this group. In this
paper we work with infinite-dimensional unitary group $U(\infty)$
that is the union of classical groups $U(N)$ naturally embedded one
into another. Since $U(\infty)$ is not a locally compact group, the
conventional definition of the regular representation is not
applicable to it. In the pioneer paper \cite{Pickrell} Pickrell in a
similar situation proposed a construction to deal with this
difficulty. It was the starting point of \cite{Olsh} and \cite{BO}
where Borodin and Olshanski introduced a natural generalization of
the regular representation for the case of the ``big'' group
$U(\infty)$.

The space of the representations introduced by Borodin and Olshanski
is an inductive limit of the spaces $L_2(U(N))$. In order to
correctly define this inductive limit one has to choose some
embeddings $L_2(U(N))\hookrightarrow L_2(U(N+1))$. It turns out that
this choice is not unique. Different possible choices lead to a
family $T_{z,w}$ of the representations which depend on two complex
parameters $z,w$. The representation $T_{z,w}$ does not change if
$z$ or $w$ is replaced by $\overline z$ or $\overline w$,
respectively.  The structure of $T_{z,w}$ substantially depends on
whether parameters are integers or not. In the present paper we
handle the latter case.

Since all representation $T_{z,w}$ are inductive limits of the very
same sequence of spaces a very natural question arises: Are these
representations really different? Our paper provides an answer. We
prove that representations $T_{z,w}$ are disjoint. Recall that two
representations $T$ and $T'$ are called disjoint if they have no
equivalent nonzero subrepresentations. In this paper we prove the
following theorem:

\begin{theorem}
 Suppose that all parameters $z,w,z',w'$ are not integers,
 $\{z,\overline z\} \neq \{z',\overline {z'}\}$,
 $\{w,\overline w\}\neq\{w',\overline {w'}\}$ (all pairs are non-ordered),
 $\Re(z+w)>-1/2$,
 then representations
 $T_{z,w}$ and $T_{z',w'}$ are disjoint.
\end{theorem}

Decomposition of the representations $T_{z,w}$ into irreducible ones
is governed by some measures, we call them spectral measures.
Instead of the disjointness of the representations $T_{z,w}$ we may
prove the disjointness of the corresponding spectral measures. These
measures were fully described in the paper \cite{BO}. There exist
explicit formulas for the correlation functions of the spectral
measures. But, as far as the author knows, no direct ways to derive
the disjointness of measures looking at their correlation functions
are known. Thus, we have to use different approach here. We reduce
the disjointness of spectral measures to disjointness of certain
probability measures $\rho_{z,w}$ on the space of paths in the
Gelfand-Tsetlin graph, and we are mainly concerned with analysis of
the measures $\rho_{z,w}$.

Disjointness of representations and measures connected with ``big''
groups was discussed in numerous papers and there are still many
related problems which remain unsolved.

The authors of \cite{KOV} studied a family $T_z$ of representations
of the infinite symmetric group and obtained a result similar to
Theorem 1. Our reasoning uses some ideas from that paper, but the
case of the unitary group turned out to be much more complicated and
we have to use additional arguments.

Another closely related example is given in \cite{BO_Herm} where
Borodin and Olshanski proved the disjointness of the so-called
Hua-Pickrell $s$--measures on the space of infinite Hermitian
matrices.

In the paper \cite{Ner} Neretin generalized original Pickrell
construction to all ten classical series of compact Riemannian
symmetric spaces. The problem of the disjointness can be posed in
each of these ten cases. Perhaps, the disjointness can be proved
using methods similar to the ones of our paper.

Note that in Theorem 1 we have a restriction on the parameters
$z,w$: $\Re(z+w)>-1/2$. Actually, the construction of the
representations $T_{z,w}$ makes sense for all $z,w\in\mathbb C$. But
for $\Re(z+w)\le -1/2$ we get representations without so-called
distinguished $K$--invariant vector. It would be interesting to
study these representations too, but our methods are not applicable
for them.

In the second part of our paper we present a result of a different
kind: A path in the Gelfand-Tsetlin graph is a sequence $\{\lambda
(N)\}$ of dominant weights of unitary groups $U(N)$, and each
$\lambda (N)$ can be viewed as pair $(\lambda^+(N),\lambda^-(N) )$
of Young diagrams. Thus, each measure $\rho_{z,w}$ generates two
random sequences $\{\lambda^+(N)\}$,$\{\lambda^-(N)\}$ of Young
diagrams forming a Markov growth process. As an application of our
method we prove that the length of the diagonal in
$\{\lambda^\pm(N)\}$ grows at most logarithmically in $N$.

The author would like to thank G.~Olshanski for suggesting the
problem, numerous fruitful discussions and help in simplifying
proofs. Also the author is grateful to A.~Borodin who has found a
gap in the original proof. The author thanks B.~Gurevich and
S.~Pirogov  for helpful discussions and advice.

The author was partially supported by RFBR grant 07-01-91209, the
Moebius Contest Foundation for Young Scientists and Leonhard Euler's
Fund of Russian Mathematics Support.

\section{Representations of $U(\infty)$, Gelfand-Tsetlin graph, reduction to spectral measures}
In this section we collect some general facts about representations
of the group $U(\infty)$; most of them can be found in the paper
\cite{Olsh}.

Consider the chain of the compact classical groups $U(N)$,
$N=1,2,\dots$, which are embedded one into another in a natural way.
Let $U(\infty)$ be their union. Following the philosophy of
\cite{Olsh_Howe}, we form a $(G,K)$--pair, where $G$ is the group
$U(\infty)\times U(\infty)$ and $K$ is the diagonal subgroup,
isomorphic to $U(\infty)$. We are dealing with unitary
representations $T$ of the group $G$ possessing a distinguished
cyclic $K$--invariant vector $\xi$. Such representations are called
\emph{spherical representations} of the pair $(G,K)$. They are
completely determined by the corresponding matrix coefficients
$\psi(\cdot)=(T(\cdot)\xi,\xi)$. The $\psi$'s are called
\emph{spherical functions}. They are $K$--biinvariant functions on
$G$, which can be converted (via restriction to the subgroup
$U(\infty)\times\{e\}\subset G$) to central functions $\chi$ on
$U(\infty)$ (that are functions which are constant on conjugacy
classes). Irreducible representations $T$ are in one-to-one
correspondence to \emph{extreme characters} (i.e. extreme points in
the convex set of all characters, which are central positive
definite continuous functions on $U(\infty)$ taking on a value of
$1$ at the unity element of the group).

Irreducible spherical representations of $(G,K)$ and extreme
characters of $U(\infty)$ admit a complete description. They depend
on countably many continuous parameters. There is a bijective
correspondence $\chi^{(\omega)}\leftrightarrow \omega$ between
extreme characters and points $\omega$ of the infinite-dimensional
domain
$$
 \Omega\subset{\mathbb R}^{4\infty+2}={\mathbb R}^\infty\times {\mathbb R}^\infty
  \times {\mathbb R}^\infty \times {\mathbb R}^\infty \times {\mathbb R} \times {\mathbb R},
$$
where $\Omega$ is the set of sextuples

$$
\omega=(\alpha^+,\beta^+;\alpha^-,\beta^-;\delta^+,\delta^-)
$$
such that
$$
 \alpha^\pm=(\alpha_1^\pm\ge\alpha_2^\pm\ge\dots\ge 0)\in {\mathbb R}^\infty,
 \quad \beta^\pm=(\beta_1^\pm\ge\beta_2^\pm\ge\dots\ge 0)\in {\mathbb R}^\infty,
$$
$$
 \sum\limits_{i=1}^\infty(\alpha_i^\pm+\beta_i^\pm)\le\delta^\pm,\quad \beta_1^+ +\beta_1^-\le 1.
$$
D.~Voiculescu discovered an explicit formula for the functions
$\chi^{(\omega )}(U)$, where $U\in U(\infty)$. We do not need it
 in the present paper, this formula can be found
in \cite{Vo} or \cite{Olsh}.

It was proved in \cite{Olsh} that any character $\chi$ corresponds
to a unique probability measure $\sigma$ on $\Omega$, such that
$$
 \chi(U)=\int_\Omega \chi^{(\omega)}(U)\sigma(d\omega), \quad U\in U(\infty).
$$
We call $\sigma$ the \emph{spectral measure} of the character ${\chi}$. The
inverse statement is also true, every probability measure on
$\Omega$ corresponds to a certain character of the group
$U(\infty)$.

According to the general theory, the disjointness of two spherical
representations $T$ and $T'$ is equivalent to the disjointness of
the spectral measures $\sigma$ and $\sigma'$ corresponding to their
characters. Our aim is to prove the disjointness of these measures.

\emph{The Gelfand-Tsetlin graph} (also known as the graph of signatures) is
a convenient tool for describing characters of $U(\infty)$. The
vertices of the graph symbolize the irreducible representations of
groups $U(N)$ while the edges encode the inclusion relations between
irreducible representations of $U(N)$ and $U(N+1)$. The $N$--th
level of the graph, denoted by $\mathbb{ GT}_N$, corresponds to the
irreducible representations of $U(N)$. Elements of $\mathbb{GT}_N$
can be identified with dominant weights for $U(N)$, i.e., these are
$N$--tuples of integers $\lambda=(\lambda_1\ge\dots\ge\lambda_N)$
which are also called \emph{signatures}. We join two signatures $\lambda
\in \mathbb{GT}_N$ and $\mu\in\mathbb{GT}_{N+1}$ by an edge and
write $\lambda\prec\mu$ if
$$\mu_1\ge\lambda_1\ge\mu_2\ge\dots\ge\lambda_N\ge\mu_{N+1}.$$

In what follows it is convenient for us to represent signature
$\lambda$ by two Young diagrams $\lambda^+$ and $\lambda^-$: the row
lengths of $\lambda^+$ are the positive coordinates $\lambda_i$,
while the row lengths of $\lambda^-$ are the absolute values of the
negative coordinates. We call $\lambda^+$ and $\lambda^-$ \emph{a
``positive'' Young diagram} and a \emph{``negative'' Young diagram},
respectively.

If $\tilde \lambda$ and $\tilde \mu$ are two Young diagrams such
that $\tilde \lambda\subset\tilde \mu$ then their set-theoretical
difference is called \emph{skew diagram} and we denote it by $\tilde
\lambda / \tilde \mu$ (see \cite[Chapter 1]{Mac}). Recall that a
 horizontal strip is a skew Young diagram containing at most one box in each column.
Note that $\lambda\prec\mu$ if and only if both $\mu^+/\lambda^+ $
and $\mu^-/\lambda^-$ are \emph{horizontal strips}.

Recall that number $c(\square)=j-i$ is called the \emph{content} of
the box $\square$ lying  at the intersection of the $i$--th row and
$j$--th column.

We denote by $\chi^\lambda$ the irreducible character of $U(N)$,
indexed by the signature $\lambda\in\mathbb{GT}_N$. Given an
arbitrary character $\chi$ of the group $U(\infty)$ we may expand
the restriction of $\chi$  on
$U(N)$ into a convex combination of the functions
$\chi^\lambda(\cdot)/\chi^\lambda(e)$.
 The coefficients of this
expansion determine a probability distribution $M_N(\lambda)$ on the
discrete set $\mathbb{GT}_N$. In this way, we get a bijection
between characters and certain sequences of probability
distributions $\{M_N\}_{N=1,2,\dots}$, which are called
\emph{coherent systems} (because of certain coherency relations
connecting $M_N$ and $M_{N+1}$). Additional information can be found
in \cite{Olsh}.

\emph{A path} in the graph of signatures $\mathbb{GT}$ is a finite
or infinite sequence $t=(t(1),t(2),\dots)$, such that
$t(N)\in\mathbb{GT}_N$, and for every $N=1,2,\dots$, $t(N)$ and
$t(N+1)$ are joined by an edge. In what follows finite or infinite
paths are denoted by letters $t$, $\tau$. Usually
$\tau=(\tau(1),\tau(2),\dots)$ and $\tau_i(N)$ stands for the length
of the $i$-th row of the signature $\tau(N)$.

We say that a box $\square$ is added to the ``positive'' Young
diagram at level $N$ while moving along a path $\tau$ if
$\square\in\tau^+(N+1)/\tau^+(N)$.

 We can define a topology on the set
$\mathcal T$ of all infinite paths in $\mathbb{GT}$. A base of this topology
consists of \emph{cylindrical sets} $C_\tau$, where
$\tau=(\tau(1),\dots,\tau(N))$ is a finite path and
$$C_\tau=\{t\in {\mathcal T} | t(1)=\tau(1),\dots,t(N)=\tau(N)\}.$$

Now we may take $\sigma$--algebra of Borel sets generated by this
topology and consider measures on this $\sigma$--algebra.

A measure on $\mathcal T$ is called \emph{central} if the measure of a
cylindrical set $C_\tau$, $\tau=(\tau(1),\dots,\tau(N))$ depends
solely on $\tau(N)$.

Consider an arbitrary signature $\lambda\in\mathbb{GT}_N$. The set
of finite paths terminating at $\lambda$ consists of exactly
$\operatorname{Dim}_N(\lambda)=\chi^\lambda(e)$ elements. Given a
coherent system on the graph of signatures we may construct a
certain central measure $\rho$ on the set $\mathcal T$ by setting
$$\rho(C_\tau)=\frac{M_N(\tau(N))}{\operatorname{Dim}_N(\tau(N))}$$
where $C_\tau$ is an arbitrary cylindrical set as above and
$\tau(N)$ is the endpoint of $\tau$. Correctness of this definition
follows from the properties of coherent systems.

Thus, we obtain a sequence of bijections connecting characters of
the group $U(\infty)$, probability measures $\sigma$ on $\Omega$
(spectral measures),
 coherent systems $\{M_N\}$, and
central measures $\rho$ on $\mathcal T$.

Let us show that the disjointness of spectral measures $\sigma$ and
$\sigma'$ follows from the disjointness of the corresponding central
measures $\rho$ and $\rho'$ (the similar proposition in the case of
the infinite symmetric group was proved in \cite{KOV}).

\begin{proposition}
Assume that central measures $\rho$ and $\rho'$ on the graph of
signatures are disjoint, then the corresponding spectral measures $\sigma$ and $\sigma'$ on
$\Omega$ are also disjoint.
\end{proposition}
\begin{proof}
 Assume the contrary. Denote by $\tilde\sigma=\sigma\wedge\sigma'$
 the greatest lower bound of the measures $\sigma$ and $\sigma'$ (we say that a measure
 $\mu$ is less than a measure $\nu$ if
 $\mu (A) \le \nu (A)$ for an arbitrary measurable set $A$). The existence of such measure can be easily verified.
 It is evident that if $\sigma$ and $\sigma'$ are not disjoint, then $\tilde\sigma$ is a non-zero measure.

 Next, observe that the correspondence $\sigma\leftrightarrow\rho$ between probability spectral measures
 on $\Omega$ and central measures on $\mathcal T$ can be extended to arbitrary finite measures, which are not necessarily
  probability measures.

 Denote by $\tilde \rho$ the central measure corresponding to $\tilde\sigma$.

 There is an integral representation for the value of this measure on a cylindrical
 set. Let $\omega\in\Omega$. Let $\rho_\omega$ be the central measure on $\mathcal T$ corresponding to the extreme
 character $\chi^{(\omega )}(U)$.
 Denote by $f_{\tau(N)}(\omega)$ the value of $\rho_\omega$ on the
cylindrical set $C_\tau$, where $\tau(N)$ is the endpoint of $\tau$.
We have
 $$
  \tilde\rho(C_\tau)=\int_\Omega
f_{\tau(N)}(\omega)\tilde\sigma(d\omega).
 $$

 From the last formula, nonnegativity of
 $f_\lambda(\omega)$, and inequalities
 $\tilde\sigma\le\sigma$, $\tilde\sigma\le\sigma'$ it follows that
 $\tilde\rho\le\rho$ and $\tilde\rho\le\rho'$. Consequently,
 $\rho\wedge\rho'\neq 0$ and the measures $\rho$ and $\rho'$ are not disjoint.

\end{proof}

In what follows we are going to prove disjointness of the central
measures on the graph of signatures corresponding to the characters
of the representations $T_{z,w}$ and $T_{z',w'}$. Let us denote
these measures by $\rho_{z,w}$ and $\rho_{z',w'}$.

It was shown in \cite{Olsh} that the coherent system corresponding
to the representation $T_{z,w}$ is given by
$$
M_N^{z,w}(\lambda)=\left( S_N(z,w) \right) ^{-1}\cdot \prod\limits_{i=1}^N
\left|\frac{1}{\Gamma(z-\lambda_i+i)\Gamma(w+N+1+\lambda_i-i)}\right|^2\cdot
\operatorname{Dim}^2_N(\lambda),
$$
where $ S_N(z,w) $ is the normalization constant whose explicit
value is not important for us. Denote by $P_N^{z,w}(\lambda)$ the
measure $\rho_{z,w}$ of an arbitrary cylindrical set $C_\tau$, such
that $\lambda$ is the endpoint of the path $\tau$.
 The following
formula holds:
$$\rho_{z,w}(C_\tau)=P_N^{z,w}(\tau(N))=M_N^{z,w}(\tau(N))/\operatorname{Dim}_N(\tau(N)).$$

The measure $\rho_{z,w}$ is called the $(z,w)$--measure below.

From this point on we forget about representations $T_{z,w}$, all
further arguments are based on the last two formulas.

\section{Scheme of proof}

Our aim is to prove that the $(z,w)$--measures are pairwise disjoint.

Consider two pairs $(z,w)$ and $(z',w')$, assume that all four
numbers are non-integral, $z\neq z'$, $z\neq \overline{z'}$, $w\neq
w'$, $w\neq \overline{w'}$. Let us study the ratio
$\frac{P_N^{z,w}}{P_N^{z',w'}}$. Let ${\mathfrak Q}$ be the set of
paths along which this ratio tends to a finite non-zero limit.
Clearly, this is a Borel set. The disjointness of the measures
follows from the equalities  $\rho_{z,w}({\mathfrak
Q})=\rho_{z',w'}({\mathfrak Q})=0$ and the following proposition

\begin{prop}[ {\cite[Chapter 7, Section 6, Theorem 2]{Shir}} ]
Let $X$ be a measurable space with a filtration $\{A_n\}$. Denote by
$P_n$ and $\tilde P_n$ restrictions of two arbitrary probability
measures $P$ and $\tilde P$ on the $\sigma$--algebra $A_n$. Suppose
that $\tilde P_n$ is absolutely continuous with respect to the
measure $P_n$ and Radon-Nickodim derivative $\frac{d\tilde P_n}{d
P_n}$ tends to zero (as $n$ tends to $\infty$) almost everywhere
with respect to $P$. Then $P$ and $\tilde P$ are disjoint.
\end{prop}

In order to show that the measure of the set ${\mathfrak Q}$ is
equal to zero, we introduce an auxiliary set $G$ (it will be defined
later). Denote temporarily ${\mathfrak Q_0}=G\cap {\mathfrak Q} $.
We prove below that the measure of ${\mathfrak Q_0}$ is equal to
zero (Sections 3 and 4) and then we show that the $(z,w)$--measure
of the set $G$ is equal to $1$ (Section 5). Clearly, these two facts
imply vanishing of the $(z,w)$--measure of the set ${\mathfrak Q}$.

Proof of the first part (the estimate for the measure of ${\mathfrak
Q_0}$) contains several technical details and we want to provide
some main ideas first. It is possible to pass from ${\mathfrak Q_0}$
to a family of sets ${\mathfrak Q_0^l}$. For every ${\mathfrak
Q_0^l}$ we construct a collection of maps $\{f_i\}$, which take
${\mathfrak Q_0^l}$ into a collection $\{{\mathfrak Q_i^l}\}$ of
pairwise disjoint sets. The measure of each ${\mathfrak Q_i^l}$ is
equal to the measure of ${\mathfrak Q_0^l}$. Since there exist
infinitely many such sets ${\mathfrak Q_i^k}$ we conclude that the
measure of each one is equal to zero. Hence, the measure of
${\mathfrak Q_0}$ is also equal to zero.

The essence of the maps $f_i$ is a small reorganization of paths. It
turns out, that it is enough to change only one level of a path to
generate a rather big fluctuation from the hypothetical limit value of
the ratio $\frac{P_N^{z,w}}{P_N^{z',w'}}$. This observation
is the main idea underlying the proof of the theorem.

Each map $f_i$ is very simple. We fix in advance some large integer
$k$. Each $f_i$ modifies every path as follows: Instead of adding a
box with content $k$ to the ``positive'' Young diagram at level $N$,
we do that at level $N+1$. Although, not every path can be modified
in this way, we will show later that for almost all paths from the
set ${\mathfrak Q_0^l}$ maps $f_i$ can be correctly defined.

Important moment in the argument is the following: the probability
of adding a horizontal strip containing $m$ boxes with the contents
$k,k+1,\dots ,k+m-1$ at level $N$ does not exceed $c(k,m)/N^m$ (we
use this statement only in the case $m=1$, but it also holds in the
more general case).

\section{Modification of paths}

We proceed to the detailed proof.

Let $\tau=(\tau(1),\tau(2),\dots)$ be a path in the graph of
signatures. $\tau^+(N)$ is an increasing sequence of Young diagrams.
Denote $\overline\tau^+=\bigcup\limits_{N}\tau^+(N)$. It is clear
that $\overline\tau^+$ is either the whole quadrant or only some its
part called \emph{a thick hook}. Denote by $G^+$ the set of all
paths $ \tau$ such that $\overline\tau^+$ is the whole quadrant. In
a similar way we define
$\overline\tau^-=\bigcup\limits_{N}\tau^-(N)$ and $G^-$ is the set
of all paths $\tau$ such that $\overline\tau^-$ is the whole
quadrant. In Section 5 we show that $\rho_{z,w}(G^+\cup G^-)=1$. Now
our aim is to prove the following theorem:

\begin{theorem}
 If $\{z, \overline{z}\}\neq\{z', \overline{z'}\} $, then
 $$\rho_{z,w}({\mathfrak Q}\cap G^+)=\rho_{z',w'}({\mathfrak Q}\cap G^+)=0.
 $$
 If $\{w, \overline{w}\}\neq\{w', \overline{w'}\} $, then
 $$\rho_{z,w}({\mathfrak Q}\cap G^-)=\rho_{z',w'}({\mathfrak Q}\cap G^-)=0.
 $$
\end{theorem}

In particular this theorem implies that if $z\neq z'$, $z\neq
\overline{z'}$, $w\neq w'$ and $w\neq \overline{w'}$, then the
corresponding measures are disjoint.

Note that both Gelfand-Tsetlin graph and $z,w$--measures are
invariant under the symmetry
$$
 z\leftrightarrow w,\quad
(\lambda_1,\dots,\lambda_N)\leftrightarrow(-\lambda_N
,\dots,-\lambda_1).$$ Thus, it is enough to prove only the first
part of the theorem.

Let us denote by $h_N(\tau)$ the double ratio
$$h_N(\tau)=\frac{P_{N+1}^{z,w}(\tau(N+1))}{P_{N+1}^{z',w'}(\tau(N+1))} :
\frac{P_N^{z,w}(\tau(N))}{P_N^{z',w'}(\tau(N))}.$$

Note that if $\frac{P_N^{z,w}(\tau(N))}{P_N^{z',w'}(\tau(N))}$ tends
to a finite nonzero limit, then $h_N(\tau)\to 1$.

Denote by ${\mathfrak A}$ the set of paths along which  $h_N\to 1$.
Our aim is to prove that the $(z,w)$--measure of the set ${\mathfrak
A}\cap G^+$ is equal to zero ($\rho_{z,w}({\mathfrak A}\cap
G^+)=0$).

Fix $\delta >0$. We say that there is a $\delta$--fluctuation, or
simply a fluctuation, at the level $k$ provided that
${|h_k-1|>\delta}$. Denote by $A_t^\delta$ the set of those paths
that contain no $\delta$--fluctuations starting from level $t$. We
will show that if we choose sufficiently small $\delta$, then
$\rho_{z,w}(A_t^\delta \cap G^+)=0$ for every $t$. It is clear that
this statement implies our theorem.

\begin{theorem} There exists $\delta>0$ such that for every $t$
 $$\rho_{z,w}(A_t^\delta \cap G^+)=0.$$
\end{theorem}
\begin{proof}

Recall once again that the content of the box lying at the
intersection of the $i$-th row and the $j$-th column of a Young
diagram is the integer $c(\square)=j-i$.

It is evident that for any path from $G^+$ and for every $k$ a box
with the content $k$ is added to the ``positive'' Young infinitely
many times.

The idea of the subsequent argument is the following: adding a box
with a relatively small positive content may cause a
$\delta$--fluctuation. Suppose $\overline\tau$ is a finite fragment
of a path, there are no $\delta$--fluctuations in $\overline\tau$,
and a box with the content $k$ is added while moving along
$\overline\tau$. Then for almost any $\overline\tau$ there exists
another fragment $\overline\tau'$, such that $\overline\tau$ and
$\overline\tau'$ have the same starting point and endpoint, and
there is a $\delta$--fluctuation in $\overline\tau'$. Since
infinitely many boxes with the content $k$ are being added along the
typical path, the last claim also implies a stronger assertion. If a
fragment $\overline\tau$ is long enough, than we may construct
arbitrary many such fragments $\overline\tau'$. In turn, this
implies that measure of the set of paths without
$\delta$--fluctuations is equal to zero.

So fix a positive integer $k$, its value will be specified later. We
introduce a family of sets of paths $B(k,\{N_i\})$. In what follows
we construct certain maps $f_i$ on these sets which ``create
$\delta$--fluctuations''.

Let us fix some integer $k$ and a sequence of positive integers
$N_1<N_2<\dots$. Denote by $B(k,\{N_i\})$ a set of all paths
$\tau\in G^+$ such that for every $m>0$ there exist a positive
integer $n(m)$ and a box $\square(m)$ with coordinates $(i,j)$ which
meet the following conditions:
 \begin{enumerate}
  \item $n(m)$ is a minimal number $n\in [N_m,N_{m+1})$ such that
$\tau^+(n+1))/\tau^+(n)$ contains a box with content $k$
  \item $\square(m)\in \tau^+(n(m)+1))/\tau^+(n(m))$ and
$c(\square(m))=j-i=k$

  \item $\tau^+(n(m)+1))/\tau^+(n(m))$ does not contain a box with content $k+1$ (i.e. it does not contain a box with coordinates $(i,j+1)$ )
  \item $\tau^+(n(m)+2))/\tau^+(n(m)+1)$ does not contain a box with content $k-1$ (consequently, it does not contain
  a box with coordinates $(i+1,j)$)
 \end{enumerate}

\begin{proposition}[about admissible paths]
 Assume that $k$ is so large that $\Re (k+w)>0$ then for an arbitrary $\varepsilon > 0$ there exists a sequence of integers
 $N_1<N_2<\dots$ such that
  $\rho_{z,w}(G^+ \setminus B(k,\{N_i\})<\varepsilon$.
\end{proposition}

We give the proof of this proposition in the next section. Now we
are going to use it to prove Theorem 4.

Denote by $H$ the intersection of the sets $B(k,\{N_i\})$ and
$A_t^\delta$. Clearly, it is sufficient to prove that for small
enough $\delta$ the measure of $H$ is equal to zero.

A path from the set $H$ can be informally described in the following
way. Starting from some level it has no $\delta$--fluctuations but
boxes with content $k$, which were defined in the proposition, are
still added to the ``positive'' Young diagram.

Now fix $H$ and the sequence of integers $\{N_i\}$ corresponding to
it. Choose some large enough number $N_p$ from the $N_i$'s. Choose
an integer $q>p$ and consider the levels from $N_p$ to $N_q$ of the
graph of signatures. We call a sequence
$(\tau(N_p),\tau(N_p+1),\dots,\tau(N_q))$ a
\emph{$[N_p,N_q]$-fragment} of path $\tau=(\tau(1),\tau(2),\dots)$.

Fix an arbitrary signature $\lambda^p$ at level $N_p$ and an
arbitrary signature $\lambda^q$ at level $N_q$. Consider all
$[N_p,N_q]$--fragments of paths $\tau$, such that
$\tau(N_p)=\lambda^p$ and $\tau(N_q)=\lambda^q$. There is a finite
number of such $[N_p,N_q]$--fragments. The central property of the
$(z,w)$--measures implies that these fragments have uniform
distribution. We want to study the fraction of the
$[N_p,N_q]$--fragments of paths from the set $H$ in all such
$[N_p,N_q]$--fragments. We will prove that this fraction tends to
zero when $q\to \infty$ (uniformly in $\lambda^p$ and $\lambda^q$).
Obviously, this implies that the measure of the set $H$ is equal to
zero.

Let us also choose signatures at levels $N_{p+1},N_{p+2} \dots
N_{q-1}$. Suppose that
$\tau(N_{p})=\lambda^{p},\tau(N_{p+1})=\lambda^{p+1},\dots,\tau(N_{q})=\lambda^q$
and examine $[N_p,N_q]$--fragments of such paths $\tau$. We again
want to study the fraction of such $[N_p,N_q]$--fragments, obtained
from the paths of $H$, in all such $[N_p,N_q]$--fragments. It
sufficient to prove that this fraction tends to zero, provided that
$q$ tends to infinity.

Now fix a $[N_p,N_q]$--fragment, satisfying the above conditions.
Consider a $[N_s,N_{s+1}]$--fragment, that is a subfragment of the
above $[N_p,N_q]$--fragment, and define a path modification $f_s$ in
the following way.

We know that a box with content $k$ is added at a level $n(s)\in
[N_s, N_s+1, \dots , N_{s+1})$. Modification $f_s$ consists in
adding this box not at level $n(s)$, but at level $n(s)+1$ (i.e. we
remove this box from the skew diagram $\tau^+(n(s)+1)/\tau^+(n(s))$
and add the very same box to the skew diagram
$\tau^+(n(s)+2)/\tau^+(n(s)+1)$) Conditions $3$ and $4$ on the set
$B(k,\{N_i\})$ imply that this modification is possible. Thus, we
obtained some map, its domain of definition is the set $H$. It is
clear that this map is injective. (To determine the preimage of a
path one should choose the first addition of a box with content $k$
and simply move this addition to the previous level)

\begin{lemma} It is possible to choose $k$, $\delta$ and $p$ such that
 for any path $\tau\in H$, $f_s(\tau)$ does not belong to $H$.
\end{lemma}
\begin{proof}

Consider
$$h_{n(s)}(\tau)=\frac{P_{n(s)+1}^{z,w}(\tau(n(s)+1))}{P_{n(s)+1}^{z',w'}(\tau(n(s)+1))} :
\frac{P_{n(s)}^{z,w}(\tau(n(s)))}{P_{n(s)}^{z',w'}(\tau((n(s))))}.$$
Note that $\tau(n(s))$ does not change under the map $f_s$ while $
\tau(n(s)+1)$ does. Set $N:=n(s)+1$ and examine the level $N$. The
essence of our modification is that in the new path for a single $i$
the number $\lambda_i$ (which is the length of the $i$--th row)
decreased by $1$. Recall that
$$
P_N^{z,w}(\lambda)=\left( S_N(z,w) \right) ^{-1}\cdot \prod\limits_{i=1}^N
\left|\frac{1}{\Gamma(z-\lambda_i+i)\Gamma(w+N+1+\lambda_i-i)}\right|^2\cdot
\operatorname{Dim}_N(\lambda)
$$

Consequently, under the transform $\lambda_i\mapsto\lambda_i-1$ this
density is multiplied by

$$
  \left|\frac{w+N+1+k}{z-k}\right|^2
  \cdot (\text{factors not depending on } z,w ).
$$
We conclude that the ratio of densities
$\frac{P_{N}^{z,w}}{P_{N}^{z',w'}}$ is multiplied by
$$
 \left|\frac{z'-k}{z-k}\right|^2 \cdot
 \left|\frac{w+N+1+k}{w'+N+1+k}\right|^2
 \eqno (*)
$$

Examine the first factor. Note that if $z\neq z'$ and $z\neq
\overline{z'}$, then one can find $k>0$ and $\nu>0$ such
that
$$
 \left|\left|\frac{z'-k}{z-k}\right|^2-1\right|>\nu
$$
and $\Re(k+w)>0$. Let us fix such $k$ and $\nu$.

Now consider the second factor in $(*)$. It is clear that this factor approaches $1$ as $N$ tends to
infinity. Assume that
$N_p$ is so large that for any $N>N_p$ the factor differs from $1$ by less
than $\nu/100$.

Now we can conclude that the product $(*)$ differs from $1$ by more
than $\nu/2$.

Choose $\delta \le \nu/100$. It is evident that for such $\delta$
the ratios of densities before and after applying $f_s$ can
not both belong to the $\delta$--neighborhood of $1$  (as
$f_s$ multiplies the ratio of densities by the number $(*)$ that is
far enough from $1$).

Choose $p$ such that the inequality $N_p\ge t$ also holds. From all
that has been said above it follows that in any path from $H$ there
are no $\delta$--fluctuations at  levels starting from $N_p$, but
after applying $f_s$ such $\delta$--fluctuations inevitably appear.
Thus the image with respect to $f_s$ of a path from the set $H$ does
not belong to $H$.
\end{proof}

\begin{lemma} Suppose that the parameters are chosen as in Lemma 6. If $s\neq s'$, then
the sets $f_s(H)$ and $f_{s'}(H)$ are mutually disjoint.
\end{lemma}
\begin{proof}
 It follows from the proof of the last lemma that in the paths from $f_{s}(H)$ and $f_{s'}(H)$ $\delta$--fluctuations
 appear at distinct levels, thus, these paths are distinct.
\end{proof}

Now assume that we have $E$ distinct $[N_p,N_q]$--fragments of paths
from the set $H$. Every $f_s$ produces $E$ additional fragments and
all $f_s$ together produce $E(q-p)$ distinct fragments.
Consequently, the fraction of the $[N_p,N_q]$--fragments of the
paths from the set $H$ in all $[N_p,N_q]$--fragments (hence also the
measure of the set $H$) does not exceed $\frac{E}{E+E(q-p)}$.
Choosing the number $q$ large enough we can obtain arbitrary small
values of the last expression. Thus, Theorem 4 is proved.

\end{proof}

\section{Proof of the proposition about admissible paths}
\begin{proof}[Proof of Proposition 5]
 Let us interpret the $(z,w)$--measure as a Markovian process on the Gelfand-Tsetlin graph with transition
 probabilities
$$p_{z,w}(\lambda|\mu)=\frac{P^{z,w}_{N+1}(\lambda)}{P^{z,w}_N(\mu)},\quad\lambda\in\mathbb{GT}_{N+1},\quad\mu\in\mathbb{GT}_N.$$

 Thus, from now on we may speak about the random path $\tau$ or about the Markov chain $\tau(N)$.

\begin{lemma}
 Assume $k$ is so large that $\Re (k+w)>0$. Fix an arbitrary signature  $\mu$
 at level $N$ (thus also fix the corresponding "positive" Young diagram).
 The following two estimates on conditional probability holds
 $${\rm Prob} \{\exists \square\in \tau^+(N+1)/\tau^+(N): c(\square)=k \mid \tau(N)=\mu
\}<\frac{c(k)}{N},$$
 \begin{multline*}
{\rm Prob} \{\exists \square\in \tau^+(N+1)/\tau^+(N): c(\square)=k
\mid
\\ \mid  \tau(N)=\mu, \exists \square\in \tau^+(N+1)/\tau^+(N): c(\square)=k-1
  \}<\frac{c(k)}{N},\end{multline*}

\end{lemma}
{\bf Remark.} Instead of $\mu$ we may fix the whole path $(\mu^1,\dots,\mu^{N})$ from the first level to the signature $\mu^N$,
 and replace the condition $\tau(N)=\mu$ by the condition $\tau(1)=\mu^1,\dots, \tau(N)=\mu^N.$
 The centrality property of the measure implies that this change does not affect either our statement or its proof.

\begin{proof}
 Note that for the fixed $\mu$ there is at most one box with content $k$ that can be added.
 Denote the coordinates of this box by $(i,j)$. Let us compare the conditional probability of the event
 $\tau_i(N+1)=j-1$ , given that $\tau(N)=\mu$,
 with the conditional probability of the event $\tau_i(N+1)\ge j$. We will show that the latter differs from the former by the factor
 $\frac{c(k)}{N}$, which clearly implies the lemma.

 Besides fixing signature $\mu$, fix additionally all row lengths at level $N+1$, except the $i$--th row, in an arbitrary way.
 Thus, now we condition on $\tau(N)=\mu$,
$\tau_1(N+1)=\lambda_1,\dots,\tau_{i-1}(N+1)=\lambda_{i-1},\tau_{i+1}(N+1)=\lambda_{i+1},\dots,\tau_{N+1}(N+1)=\lambda_{N+1}$.

 Denote by $p_m$ the ratio of the conditional probability of
the event $\tau_i(N+1)=j-1+m$ to the conditional probability of the
event $\tau_i(N+1)=j-1$.
 Our problem reduces to verifying the inequality $\sum\limits_{m=1}^\infty p_m < \frac{c(k)}{n}$.
 Using the definition of the transition probability and Weyl's dimension formula we obtain:
 $$ \operatorname{Dim}_{N+1}(\lambda)=
  \prod\limits_{1\le u<v \le N+1}\frac{\lambda_u-\lambda_v+v-u}{v-u},
 $$

 \begin{multline*}
 p_m=\left| \frac{\Gamma (z-k+1)\Gamma (w+N+1+k)}{\Gamma (z-k+1-m)\Gamma (w+N+1+k+m)}\right|^2\\
 \cdot
     \prod_{a=1}^{i-1}\frac{\lambda_a-a-(k-1+m)}{\lambda_a-a-(k-1)}\cdot
     \prod_{a=i+1}^{N+1}\frac{(k-1+m)-(\lambda_a-a)}{(k-1)-(\lambda_a-a)}.
 \end{multline*}
 The first product is less than $1$ and we can simply omit it, for
the second one we use the inequality
$$
 \frac{(k-1+m)-(\lambda_a-a)}{(k-1)-(\lambda_a-a)}=
\frac{m+(k-1-\lambda_a+a)}{k-1-\lambda_a+a}\le \frac{m+(a-i)}{a-i}.
$$
Consequently,
 \begin{multline*}
 p_m\le \left|\frac{\Gamma (z-k+1)\Gamma (w+N+1+k)}{\Gamma (z-k+1-m)\Gamma (w+n+1+k+m)}\right|^2
      \prod_{a=i+1}^{N+1}\frac{m+(a-i)}{a-i}\\ \le
   \left|\frac{\Gamma (z-k+1)\Gamma (w+N+1+k)}{\Gamma (z-k+1-m)\Gamma (w+N+1+k+m)}\right|^2
     \prod_{a=1}^{N}\frac{m+a}{a}\\
 =\frac{|(k-z)_m|^2}{|(w+N+1+k)_m|^2}\cdot\frac{(m+N)!}{N!m!}=
  \frac{|(k-z)_m|^2 (N+1)_m}{|(w+N+1+k)_m|^2 m!}\\
  \le \frac {\left( \left(|\Re (k-z)|+|\Im z|\right)_m\right)^2}
            {|\Re(w+N+1+k)|_m m!}\cdot
      \frac {(N+1)_m}{|\Re(w+N+1+k)|_m}
 \end{multline*}
 Note that, if $\Re(k+w)>0$, then the second factor is less than $1$, thus,
 \begin{multline*}
 p_m\le \frac {\left( \left(|\Re (k-z)|+|\Im z|\right)_m\right)^2}
            {|\Re(w+N+1+k)|_m m!}\\=
 \frac {\left( \left(|\Re (k-z)|+|\Im z|\right)\right)^2}
            {|\Re(w+N+1+k)| m }
 \frac {\left( \left(|\Re (k-z)|+|\Im z|+1\right)_{m-1}\right)^2}
            {|\Re(w+N+2+k)|_{m-1} (m-1)!}\\
 \le \frac{c(k)}{N}
  \frac {\left( \left(|\Re (k-z)|+|\Im z|+1\right)_{m-1}\right)^2}
            {|\Re(w+N+2+k)|_{m-1} (m-1)!}
 \end{multline*}
 It means that
 $$
 \sum\limits_{m=1}^\infty p_m\le \frac{c(k)}{N}
   F(|\Re (k-z)|+|\Im z|+1,|\Re (k-z)|+|\Im z|+1,
          |\Re(w+N+2+k)|;1)
 $$
 Here $F(a,b,c;1)$ is the hypergeometric function
 $_2F_1$ with parameters $a,b,c$ and the argument $1$.

 In our case $a=b=|\Re (k-z)|+|\Im z|+1$,
 $c=|\Re(w+N+2+k)|$.

 Using Gauss' formula for the value of the hypergeometric function at $1$ (see, for instance, \cite{BE})
 we obtain
 $$
 \sum\limits_{m=1}^\infty p_m\le \frac{c(k)}{N} \frac{\Gamma(c)\Gamma(c-2a)}{\Gamma(c-a)\Gamma(c-a)}
 $$
 When $N\to\infty$, the parameter $c$ tends to infinity, while $a$ does not change. Thus,
 $\frac{\Gamma(c)\Gamma(c-2a)}{\Gamma(c-a)\Gamma(c-a)}\to 1$.
 Consequently,
 $$\sum\limits_{m=1}^\infty p_m\le \frac{c_1(k)}{N}.$$

\end{proof}

\begin{lemma}
 Choose two large enough integers $m$ and $n$. Consider the following event:
 \begin{enumerate}
 \item There exist $n< l < m$ and $\square$ such that $c(\square)=k$,
$\square\in \tau^+(l+1)/\tau^+(l)$
 \item Either ${\exists \square : c(\square)=k+1,\quad \square \in \tau^+(l+1)/\tau^+(l)}$
or ${\exists \square : c(\square)=k-1,\quad \square \in
\tau^+(l+2)/\tau^+(l+1)}$.
 \end{enumerate}
The probability of this event does not exceed $\frac{c_1(k)}{n}$
\end{lemma}
\begin{proof}
 Let us represent the event ``There exist $n<l<m$ and $\square$ such that $c(\square)=k$, $\square\in \tau^+(l+1)/\tau^+(l)$''
 as the disjoint union of the events
``The minimal $n< l < m$ for which there exist $\square$ such that
$c(\square)=k$, $\square\in \tau^+(l+1)/\tau^+(l)$ is $i$''
($n<i<m$).
 Denote the last event by $L^i$.
 Let $\tau^i$ be a path going from the first level of the graph of signatures to level $i$, such that
 the minimal $n< l < m$ for which there exist $\square$ such
that $c(\square)=k$, $\square\in \tau^+(l+1)/\tau^+(l)$, is $i$.
 Then let us write $L^i$ as the disjoint union of the elementary events $C^i_{\tau^i}$ corresponding to
 the various $\tau^i$.

 Applying the previous lemma for every $C^i_{\tau^i}$ and then summing up all the estimates completes the proof.
\end{proof}

Now we continue the proof of Proposition 5.

Note that Lemma 9 implies the following: If $N_1$ is large enough
and the $N_i$'s increase at least as fast as $2^i$, then the third
condition and the fourth condition in the definition of
$B(k,\{N_i\})$ hold on a set with the measure arbitrarily close to
$1$ automatically.

Given $\varepsilon$ we choose $N_1$ so large that the measure of the
set, where then the third condition or the fourth condition in the
definition of $B(k,\{N_i\})$ might not hold, does not exceed
$\varepsilon /2$. Denote the complement of this set by $S$.

Let $A_1=S\cap G^+$. Obviously, $\rho_{z,w}(G^+\setminus A_1)\le \varepsilon/2$.

Denote by $A_2^R\subset A_1 $ ($R>N_1$) the subset containing paths
such that a box with content $k$ is added at a level from
$\{N_1,N_1+1,\dots,R\}$. It is clear that $A_2^R \subset A_2^{R+1}$.
We know that along every path from the set $A_1\subset G^+$ boxes
with content $k$ are added infinitely many times. Consequently, the
sets $A_2^R$ exhaust $A_1 $. Thus, we can choose a number
$N_2>2\cdot N_1$ such that $$\rho_{z,w}(A_1\setminus
A_2^{N_2})<\varepsilon/4.$$ Let $A_2=A_2^{N_2}$

Further we choose a number
$N_3>2\cdot N_2$ and the set $A_3\subset A_2$ in a similar way.
The inequality $\rho_{z,w}( A_2\setminus A_3)<\varepsilon/8$ holds. And so on.

The intersection of all sets $A_i$ is the desired set $B_\varepsilon$.

\end{proof}

\section{The measure of the set of paths contained in a thick hook}

We say that both ``positive'' and ``negative'' Young diagrams of a
path $\tau=(\tau(1),\tau(2),\dots)$ are contained in a thick hook if
both $\bigcup\limits_{N}\tau^+(N)$ and $\bigcup\limits_{N}\tau^-(N)$
are not the whole quadrant (i.e. these sets are so-called thick
hooks). Denote the set of all such paths by $ \overline G$. Clearly
$\overline G=\mathcal T \setminus (G^+\cup G^-)$.

\begin{theorem}
 The $(z,w)$--measure of the set $\overline G$
 is equal to zero.
\end{theorem}
\begin{proof}

 Denote by $R(i,j,l,m,t)$ the set of paths $\tau$ satisfying for every $N\ge t$
\begin{enumerate}
 \item $\tau_i(N)=j-1$,
 \item $\tau_{i-1}(N)\ge j$,
 \item $\tau_{N-l}(N)=m+1$,
 \item $\tau_{N-l+1}(N)\le m$.
\end{enumerate}
 Clearly, $\overline G$ is a union of various $R(i,j,l,m,t)$. Thus,
 to verify the claim it is sufficient to prove
 that for any quintuple $(i,j,l,m,t)$, $\rho_{z,w}(R(i,j,l,m,t))=0$ holds.

 It is convenient to view $\tau(N)$ as a Markov chain again.

\begin{lemma}
 Suppose $N$ is large enough.
Denote by $p_N$ the conditional probability of the event
``$\tau_i(N+1)=j-1$ and $\tau_{N+1-l}(N+1)=m+1$'' given that
$\tau_i(N)=j-1$, $\tau_{i-1}(N)\ge j$, $\tau_{N-l}(N)=m+1$,
$\tau_{N-l+1}(N)\le m$

Then $p_N$ does not exceed $1-\frac{\varepsilon}{N}$.
\end{lemma}
\begin{proof}
 Consider all possible signatures at level $N+1$ which are joined by an edge with
 a given signature $\mu$ at level $N$.
 For every signature $\lambda$ at level $N+1$, such that $\lambda_i=j-1$,
 we introduce a new signature $\lambda '$ as follows.
 All row lengths of $\lambda'$ except for the $i$th row, are the same as in $\lambda$, while $\lambda_i=j$.

 Let us compare the conditional probabilities of the events $\tau(N+1)=\lambda'$ and $\tau(N+1)=\lambda$,
given that $\tau(N)=\mu$.

 This conditional probabilities differ by the factor
 $$\frac{P_N^{z,w}(\lambda ')}{P_N^{z,w}(\lambda)}=
  \left|\frac{z-j+i}{w+N+1+j-i}\right|^2\cdot
  \frac{ \operatorname{Dim}_{N+1}(\lambda')}{ \operatorname{Dim}_{N+1}(\lambda).}
 $$
 If $N$ is large enough, then the first factor in the right-hand side is greater than $const/N^2$.

 By Weyl's dimension formula
 $$ \operatorname{Dim}_{N+1}(\lambda)=
  \prod\limits_{1\le i<j \le N+1}\frac{\lambda_i-\lambda_j+j-i}{j-i}.
 $$

 Consequently,
 $$
 \frac{ \operatorname{Dim}_{N+1}(\lambda')}{ \operatorname{Dim}_{N+1}(\lambda)}
 =\prod\limits_{1\le p<i}\frac{\lambda_p-p-j+i}{\lambda_p-p-(j-1)+i}\cdot
  \prod\limits_{N+1\ge p>i}\frac{j-i-(\lambda_p-p)}{j-i-1-(\lambda_p-p)}.
 \eqno(**)
 $$
 Note that here the first product is bounded from below by the constant $2^{-(i-1)}$.

 Further, all factors in the second product are greater than $1$ and
less than $2$. Moreover, if $i<p<N-l$, then $m<\lambda_p<j$.
 Thus, $\lambda_p-p$ attains all values from $\{j-i-n,\dots ,j-i-2\}$,
 except for $s$ values; and $s$  can be bounded from above by a number not depending on $N$.   We conclude that $(**)$ can be bounded
from below by the following expression:
 $$const\cdot\prod\limits_{g=j-i-2}^{j-i-N}\frac{j-i-g}{j-i-g-1}=
  const\cdot \frac{2}{1}\frac{3}{2}\dots\frac{N}{N-1}=const\cdot N$$
 Therefore, the conditional probabilities differ by a factor which is less than
  $const/N$. Consequently, $p_N\le 1-\frac{\varepsilon}{N}.$
\end{proof}

 We conclude that
 $$\rho_{z,w}(R(i,j,l,m,t))\le\prod\limits_{N=t+1}^\infty p_N\le
 \prod\limits_{N=t+1}^{\infty} \left(1-\frac{\varepsilon}{N}\right)\to 0.$$

 \end{proof}

 The author thinks that the theorem about the disjointness of the measures also holds in a more general case.

 It seems quite plausible that the following stronger statement holds:
 the $(z,w)$--measure of the set of paths contained in a left thick hook
 is equal to zero. (Recall, that a path $\tau$ is contained in a left thick hook, if there exist
 two such numbers $i$ and $j$,  that $\tau_i(N) < j$ for every $N$)

 However, at the moment there is no simple and direct proof of this statement.
 It seems likely that the statement can be verified using arguments of the
 paper \cite{BO2} about limit behavior of the intermediate Frobenius
coordinates for the "positive" and "negative" Young diagrams, corresponding to a signature (tail kernel).

 If our conjecture is true, Theorem 1 can be extended to some additional cases.
 For example, we can state that the representations
 $T_{z,w}$ and $T_{z,w'}$, where $\{w,\overline w\}\neq\{w',\overline{w'}\}$
 are also disjoint.

\section{Appendix. Growth of a diagonal}

Let $\tau$ be an infinite path in the graph of signatures.
Recall that, there exists a unique Young diagram $\lambda^+(N,\tau)$, corresponding
to the signature $\tau(N)$. We call this diagram a ``positive'' Young diagram.

As a corollary from Lemma 8, we prove that for almost every (with
respect to the $(z,w)$--measure) path the length of the diagonal of
the "positive" Young diagram grows at most logarithmically in the
number of the level.

Denote by $\tilde s_N(\tau)$ the length of the main
diagonal of $\lambda^+(N,\tau)$.

If we view
 the $(z,w)$--measure as a probability measure on the set of all paths $\mathcal T$, then
$\tilde s_N$ is a sequence of random variables.

\begin{theorem}
 There exists a positive constant $c(z,w)$ such that almost surely with respect to
the $(z,w)$--measure
 $$
  \varlimsup_{N\to\infty}\frac{\tilde s_N}{\log (N)}\le c(z,w).
 $$

\end{theorem}

\begin{proof}

First we collect some necessary general definitions and facts.

Denote by $X$ the set $\{0,1\}^\infty$ of
infinite sequences of $0$'s and $1$'s. We equip $X$ with the topology of
the direct product. Consider the sigma-algebra of Borel sets on $X$.

Given $A\in X$ denote by $A_i$ the $i$th term of the sequence $A$.
We also employ the notation $A_i=q_i(A)$.

There is a natural partial order on $X$:
$A\le B$, if $A_i<B_i$ for all $i$.

We say that a real-valued function $f$ on the space of sequences is
\emph{increasing}, if from $A\le B$ follows $f(A)\le f(B)$. Denote
by $\mathbb M$ the family of all continuous increasing functions.

Consider two probability measures $\mu_1$ and $\mu_2$ on $X$. We write $\mu_1\le\mu_2$ if
$$
 E_{\mu_1} f \le E_{\mu_2} f
$$
for all functions $f\in \mathbb M$. Here $E_\mu f$ is the expectation of the function $f$ with respect
to the measure $\mu$.

Suppose $\mathcal A\subset X$. In what follows we identify $\mathcal A$ with the event ``$x\in\mathcal A$''
and denote by  $Prob_\mu{\mathcal A}$ the probability of this event with respect to the probability measure $\mu$.
We also denote by $Prob_\mu\{\mathcal A|\mathcal B\}$ a conditional probability of the event $\mathcal A$
given the event $\mathcal B$.

\begin{proposition}
 Suppose $\mu_1$ and $\mu_2$ are two probability measures on the partial-ordered
 compact metric space $X$. A necessary and sufficient condition for  $\mu_1\le \mu_2$ is that
 there exists a probability measure $\eta$ on $X\times X$ which satisfies
 $$
  Prob_\eta\{(x,y): x\in \mathcal A \}=Prob_{\mu_1}\mathcal A,
 $$
 $$
  Prob_\eta\{(x,y): y\in \mathcal A \}=Prob_{\mu_2}\mathcal A,
 $$
 for all Borel sets $\mathcal A\subset X$, and
 $$
  Prob_\eta\{(x,y):x\le y\}=1
 $$
 The measure $\eta$ is called a coupling measure for $\mu_1$ and $\mu_2$.
\end{proposition}
\begin{proof}
 We won't need necessity of the condition, its proof can be found in the book \cite{Ligg} (theorem 2.4 of the second chapter).
 Let us proof the sufficiency of the condition.

 Let $f$ be an increasing function on $X$. The definition of the measure $\eta$
 implies that $f(x)\le f(y)$ almost surely with respect to the measure $\eta$
 on the set of pairs $(x,y)\in X\times X$. It follows that
 \begin{multline*}
  E_{\mu_1}f=\int f(x) d\mu_1=\int f(x) d\eta\le\int f(y)d\eta=\int f(y) d\mu_2=E_{\mu_2}f
 \end{multline*}

\end{proof}

\begin{lemma} Suppose $\mu$ and $\nu$ are two probability
 measures on $X$, such that $\nu$ is a Bernoulli measure (i.e. product--measure)
 and the following estimates hold
  $$Prob_{\nu}\{q_n=1\}\ge Prob_\mu\{q_n=1\mid  {\mathcal A}\},$$
  for all events $\mathcal A$, which belong to the $\sigma$--algebra generated by the first $n-1$ coordinate functions
  $q_1,\dots,q_{n-1}$.

  Then $\mu\le\nu$.
\end{lemma}
\begin{proof}
 Let us construct a coupling measure $\eta$ for $\mu$ and $\nu$.
 The support of the coupling measure $\operatorname{supp}(\eta)$ consists of pairs $(A,B)\in X\times X$, such that
 $A\le B$. It is clear that there are 3 possibilities for a pair $(A_i,B_i)$: $(0,0)$,$(0,1)$,$(1,1)$. Thus,
 $$
  \operatorname{supp}(\eta)=\{Z\}^\infty,
 $$
where $Z=\{(0,0),(0,1),(1,1)\}$.

 It suffices to define the measure $\eta$ on the cylindrical sets. The base of a cylindrical set
 is a subset of $Z^n$. Thus, we should define a family of measures $\{\eta_n\}$ ($\eta_i$ is
 a probability measure on $Z^i$) compatible with natural projections $Z^n\mapsto Z^{n-1}$.

 We proceed by induction on $n$.

 First we define the measure $\eta_1$ on $Z$.

 In other words, we want to define
 3 numbers $\eta(0,0), \eta(0,1), \eta(1,1)$ in such a way that
 \begin{equation*}
  \begin{split}
   \mu(0)=&\eta(0,0)+\eta(0,1),\\
   \mu(1)=&\eta(1,1), \\
   \nu(0)=&\eta(0,0),  \\
   \nu(1)=&\eta(0,1)+\eta(1,1). \\
  \end{split}
 \end{equation*}

 (Here we denote by $\mu(a)$ the marginal distribution $Prob_\mu\{A_1=a\}$,
 and  $\eta(a,b)=\eta_1(\{(a,b)\})=Prob_\eta\{A_1=a,B_1=b\}$)

 We may view these relations as a system of linear equations defining 3 desired numbers. It is evident,that
 as $\nu(1)\ge\mu(1)$ under the hypothesis, the system has a positive solution that is not greater than $1$.

 Now we want to define the measure $\eta_2$ on $\mathbb Z^2$.

 We know the numbers $\eta(0,0), \eta(0,1), \eta(1,1)$, and we want to
 define $9$ new numbers
 \begin{align*}
   \eta(00,00), &&\eta(00,01), &&\eta(01,01),\\
   \eta(00,10), &&\eta(00,11), &&\eta(01,11),\\
   \eta(10,10), &&\eta(10,11), &&\eta(11,11).
 \end{align*}

 Let us define the first three of them using the following relations
 \begin{equation*}
  \begin{split}
   &\eta(0,0)=\eta(00,00)+\eta(00,01)+\eta(01,01),\\
   &Prob_\mu\{X_2=0\mid X_1=0\}\cdot\eta(0,0)=\eta(00,00)+\eta(00,01),\\
   &Prob_\mu\{X_2=1\mid  X_1=0\}\cdot\eta(0,0)=\eta(01,01),\\
   &Prob_\nu\{X_2=0\mid  X_1=0\}\cdot\eta(0,0)=\eta(00,00),\\
   &Prob_\nu\{X_2=1\mid  X_1=0\}\cdot\eta(0,0)=\eta(00,01)+\eta(01,01).
  \end{split}
 \end{equation*}
 By the hypothesis,
 $$Prob_\nu\{X_2=1\mid  X_1=0\}\ge Prob_\mu\{X_2=1|\quad X_1=0\},$$
 consequently, the system has a positive, not greater than $1$, solution.

 In a similar manner we define $6$ remaining numbers using the relations

 \begin{equation*}
  \begin{split}
   &\eta(0,1)=\eta(00,10)+\eta(00,11)+\eta(01,11),\\
   &Prob_\mu\{X_2=0\mid  X_1=0\}\cdot\eta(0,1)=\eta(00,10)+\eta(00,11),\\
   &Prob_\mu\{X_2=1\mid  X_1=0\}\cdot\eta(0,1)=\eta(01,11),\\
   &Prob_\nu\{X_2=0\mid  X_1=1\}\cdot\eta(0,1)=\eta(00,10),\\
   &Prob_\nu\{X_2=1\mid  X_1=1\}\cdot\eta(0,1)=\eta(00,11)+\eta(01,11).
  \end{split}
 \end{equation*}

 and

 \begin{equation*}
  \begin{split}
   &\eta(1,1)=\eta(10,10)+\eta(10,11)+\eta(11,11),\\
   &Prob_\mu\{X_2=0\mid  X_1=1\}\cdot\eta(1,1)=\eta(10,10)+\eta(10,11),\\
   &Prob_\mu\{X_2=1\mid  X_1=1\}\cdot\eta(1,1)=\eta(11,11),\\
   &Prob_\nu\{X_2=0\mid  X_1=1\}\cdot\eta(1,1)=\eta(10,10),\\
   &Prob_\nu\{X_2=1\mid  X_1=1\}\cdot\eta(1,1)=\eta(10,11)+\eta(11,11).
  \end{split}
 \end{equation*}

 It is quite clear, the measure $\eta_2$ satisfies all required conditions.
 We derive some of them as an example
 \begin{multline*}
 Prob_\mu\{X_1=1,X_2=0\}=Prob_\mu\{X_1=1\}\cdot Prob_\mu\{X_2=0|\quad X_1=1\}=\\=
  (\eta(1,1))\cdot\left(\frac{\eta(10,10)+\eta(10,11)}{\eta(1,1)}\right)=
  \eta(10,10)+\eta(10,11).
 \end{multline*}

 \begin{multline*}
 Prob_\nu\{X_1=1,X_2=1\}=Prob_\nu\{X_1=1\}\cdot Prob_\nu\{X_2=1|\quad X_1=1\}=\\=
  (\eta(0,1)+\eta(1,1))
  \cdot\left(\frac{(\eta(10,11)+\eta(11,11))+(\eta(00,11)+\eta(01,11)) }{\eta(1,1)+\eta(0,1)}\right)=\\=
   \eta(10,11)+\eta(11,11)+\eta(00,11)+\eta(01,11).
 \end{multline*}

 The remaining conditions follow similarly from the equations defining our $9$ parameters.

 The general inductive step $n\to n+1$ is quite similar.

\end{proof}

Now we return to our theorem.

The main diagonal of a Young diagram is formed by boxes of the
diagram  such that $c(\square)=0$. Consider a secondary diagonal,
corresponding to content of a box $k$, i.e. the set of boxes with
$c(\square)=k$.

Choose $k$ such that $\Re(k+w)>0$. Denote the length of the
secondary diagonal, corresponding to the box content $k$, by
$s_{N}$. It is clear that $\tilde s_{N}\le s_{N}+k$. Therefore, it
is sufficient to prove the theorem for the random variables $s_{N}$
instead of $\tilde s_{N}$ .

 Note that since $\Re(k+w)>0$ we can apply Lemma 8
to estimate  the growth of the diagonal.

Introduce a family of random variables $\xi_N=\xi_N(\tau) $ on
$\mathcal T$. If a box with content $k$ is added to the ``positive''
Young diagram when one moves from level $N-1$ to level $N$ of the
Gelfand-Tsetlin graph along the path $\tau$ (i.e., if the length of
the secondary diagonal increases by one), then $\xi_N(\tau)=1$.
Otherwise, $\xi_N(\tau)=0$. It is evident that $s_N=\sum_{i=1}^N
\xi_i$.

The joint distribution of the random variables $\xi_i$ defines in a
natural way a probability measure on the set of sequences $X$.
Denote this measure by $\mu$.

The claim of Theorem 12 can be reformulated as a property of $\mu$,
as follows. We replace the probability space $\mathcal T$ by $X$ and
the $(z,w)$--measure by $\mu$. Then the random variables $\xi_i$
turn into the coordinate functions $q_i$ . The length of the
secondary diagonal $s_N$ turns into the sum $q_1+q_2+\dots +q_N$.

\begin{proposition}
 There exists a constant $c_1(z,w)$ such that for any collection
$\{a_1,\dots,a_{N-1}\}$ of $0$'s and $1$'s
 the following estimate on the conditional probability holds
 $$
  Prob\{\xi_N=1|\xi_1=a_1,\dots,\xi_{N-1}=a_{N-1}\}\le \frac{c_1(z,w)}{N}
 $$
\end{proposition}
\begin{proof}
 Immediately follows from Lemma 8.
\end{proof}

Consider the following family $\{\nu_N\}$ of probability measures on
$\{0,1\}$:
$$\nu_N(\{1\})= \min(1,c_1(z,w)/N),$$

$$\nu_N(\{0\})= 1-\min(1,c_1(z,w)/N).$$
 Here $c_1(z,w)$ is the constant from Proposition $15$.

Now denote by $\nu$ the direct
product of the measures $\nu_N$.

Applying Proposition $15$ and Lemma $14$ we conclude that $\mu \le \nu$.

Next, we prove that an analogue of the strong law of large numbers
holds for the measure $\nu$. Thus $\nu$ satisfies the estimate
similar to the one of Theorem $12$.

\begin{proposition}
 Almost surely with respect to the measure $\nu$
 $$
  \lim_{N\to\infty}\frac{q_1+q_2+\dots+q_N}{\log (N)}=c_1(z,w)
 $$
\end{proposition}

\begin{proof}
 It is clear that $E_\nu q_i=c_1(z,w)/i$. Consequently,
 $$
  \lim_{N\to\infty}E_\nu \left(\frac{q_1+q_2+\dots+q_N}{\log (N)}\right)
  =c_1(z,w)
 $$
 Denote $p_i=q_i-E_\nu (q_i)$.
 We have to prove that almost surely with respect to the measure $\nu$
 $$
  \lim_{N\to\infty}\frac{p_1+p_2+\dots+p_N}{\log (N)}=0
 $$
 \begin{lemma}
  Suppose $b_i$ is an arbitrary numerical sequence such that the
series
  $\sum\limits_{i=1}^\infty b_i$ converges.
  Then
 $$
  \lim_{N\to\infty}\frac{b_1\cdot\log (1)+b_2\cdot\log (2)+\dots+b_N\cdot\log (N)}{\log (N)}=0
 $$
 \end{lemma}
 \begin{proof}
  We use the discrete Abel transform:
  $$
   \sum_{k=1}^M u_k v_k = \left(\sum_{i=1}^M u_i\right) v_M+ \sum_{k=1}^{M-1}
    \left(\sum_{i=1}^k u_i\right) (v_k-v_{k+1}).
  $$
  Choose $\varepsilon>0$ and apply the last formula for $M=N$, $u_i=b_i$ and $v_i=\frac{\log (i)}{\log (N)}$.
 We obtain
  \begin{multline}
  \frac{b_1\cdot\log (1)+b_2\cdot\log (2)+\dots+b_N\cdot\log (N)}{\log (N)}=\\=
   \left(\sum_{i=1}^N b_i\right) + \sum_{k=1}^{N-1}
    \left(\sum_{i=1}^k b_i\right) \left(\frac{\log (i) - \log (i+1)}{\log (N)}\right)=\\
   =\left(\sum_{i=1}^N b_i\right)+\sum_{k=L}^{N-1}
    \left(\sum_{i=1}^N b_i\right) \left(\frac{\log (i) - \log (i+1)}{\log (N)}\right)
    -\\- \sum_{k=L}^{N-1}
    \left(\sum_{i=k+1}^N b_i\right) \left(\frac{\log (i) - \log (i+1)}{\log (N)}\right)
     +\sum_{k=1}^{L-1}
    \left(\sum_{i=1}^k b_i\right) \left(\frac{\log (i) - \log (i+1)}{\log (N)}\right)=\\=
    \left(\sum_{i=1}^N b_i\right)\left(1+\sum_{k=L}^{N-1}\frac{\log (i)- \log(i+1)}{\log (N)}\right)
    -\\- \sum_{k=L}^{N-1}
    \left(\sum_{i=k+1}^N b_i\right) \left(\frac{\log (i) - \log (i+1)}{\log (N)}\right)
     +\sum_{k=1}^{L-1}
    \left(\sum_{i=1}^k b_i\right) \left(\frac{\log (i) - \log (i+1)}{\log (N)}\right)=\\=
    \left(\sum_{i=1}^N b_i\right)\left(\frac{\log (L)}{\log (N)}\right)
    - \sum_{k=L}^{N-1}
    \left(\sum_{i=k+1}^N b_i\right) \left(\frac{\log (i) - \log (i+1)}{\log (N)}\right)
     +\\+ \sum_{k=1}^{L-1}
    \left(\sum_{i=1}^k b_i\right) \left(\frac{\log (i) - \log (i+1)}{\log (N)}\right)
  \end{multline}
  By the hypothesis the series $\sum\limits_{i=1}^\infty b_i$ converges, hence, we can choose such $L$
 that for every $m,n>L$ one has
   $\left|\sum\limits_{i=m}^{n}b_i\right|\le \varepsilon/6$. This fact implies that
  the second term in the sum $(1)$ does not exceed $\varepsilon/3$ in absolute value.
   If we now choose $N$ large enough, then the first and the second terms do not exceed $\varepsilon/3$ too.
   Consequently, the sum can be bounded by $\varepsilon$.
  \end{proof}

 Introduce $t_i=\frac{p_i}{\log (i)}$.
 The last lemma implies that it is sufficient to prove almost sure convergence
 of the series $\sum\limits_{i=1}^\infty t_i$.
 Let us estimate the variance of $t_i$. We denote the variance by $D_\nu$.
 Evidently, $D_\nu q_i\le c/i$. Therefore,
 $D_\nu t_i\le \frac{c}{i \log^2(i)}$.
 Hence, the series $\sum\limits_{i=1}^\infty D_\nu t_i$ converges.
 Now we use the following lemma:
 \begin{lemma}
  Suppose $\phi_i$ is a sequence of independent centered random
variables such that $\sum\limits_{i=1}^\infty D \phi_i$ converges.
  Then the series $\sum\limits_{i=1}^\infty \phi_i$ converges almost surely.
 \end{lemma}
  See \cite[Chapter 4, Section 2, Theorem 3]{Shir}.

  This concludes the proof of Proposition $16$.
\end{proof}
 It follows from Proposition 16 that for almost every with respect to the measure $\nu$
 sequence in $X$, there exists such $N$ that for every $n>N$:
 $$
  \frac{\sum\limits_{i=1}^n q_i}{\log (n)}<c, \eqno (***)
 $$
 where $c$ is a constant depending solely on $z,w$, and  $q_i$ is the $i$th coordinate function.
 Now we show that this statement also holds for the measure $\mu$.

 Let $\mathcal A_N\subset X$ be the set where the inequality $(***)$ holds for all $n>N$.
 It is clear that for every $N$,  $\mathcal A_N\subset \mathcal A_{N+1}$.
 By Proposition 16 $$\lim\limits_{N\to\infty} Prob_{\nu}\mathcal A_N=1.$$

 \begin{lemma}
  $$Prob_{\mu}\mathcal A_N\ge Prob_{\nu}\mathcal A_N$$
 \end{lemma}
 \begin{proof}
  Denote by $\mathcal B_N^i\subset X$ the set where inequality $(***)$ holds for $n=N,N+1,\dots,N+i-1$.
  The set $\mathcal A_N$ coincides with the intersection of the decreasing sequence of the sets $\mathcal B_N^i$
  $$\mathcal A_N=\bigcap_{i=1}^\infty \mathcal B_N^i.$$

  Note that if the inequality $(***)$ holds for some sequence $A$ and $B\le A$ (with respect to the partial order
  on $X$ defined above), then the inequality also holds for $B$. Consequently, the indicator of the set
  $\mathcal B_N^i$ is a decreasing function. Thus, $\nu\ge\mu$ implies
that $Prob_{\mu}\mathcal B_N^i\ge Prob_{\nu}\mathcal B_N^i$.
  Therefore
  $$Prob_{\mu}\mathcal A_N=\lim_{i\to\infty}Prob_{\mu}\mathcal B_N^i\ge\lim_{i\to\infty}Prob_{\nu}\mathcal B_N^i=Prob_{\nu}\mathcal A_N.$$
 \end{proof}
 It follows from the lemma that
  $$\lim_{N\to\infty} Prob_{\mu}\mathcal A_N=1.$$
 Hence, for almost every with respect to the measure $\mu$, sequence in $X$ there exists a number $N$
 such that for every $n>N$  $(***)$ holds.
 The last statement is equivalent to Theorem $12$.

\end{proof}

\end{document}